\documentclass[12pt,reqno]{amsart}
\usepackage[a4paper,bindingoffset=0.1in,left=1.2in,right=1.3in,top=1.4in,bottom=1.5in,footskip=.25in]{geometry}
\usepackage{indentfirst,amssymb,amsfonts,amsmath,amsthm}    
\usepackage{newtxtext,newtxmath}
\usepackage{setspace}
\usepackage{times}
\usepackage[utf8]{inputenc}
\usepackage[T1]{fontenc}
\usepackage{verbatim}
\usepackage{hyperref}
\hypersetup{colorlinks=true, linkcolor=blue,citecolor=blue, urlcolor=blue}
\usepackage{mathtools}
\numberwithin{equation}{section}

\newtheorem{thm}{Theorem}[section]

\newtheorem{prop}{Proposition}[section]
\newtheorem{lem}{Lemma}[section]
\newtheorem{example}{Example}[section]
\newtheorem{rema}{Remark}[section]
\newtheorem{coro}{Corollary}[section]
\newtheorem*{theorem_A}{Theorem 1}

\newtheorem*{theorem_A'}{Theorem QNH-3}
\newtheorem*{theorem_B'}{Theorem QNH-4}
\newtheorem*{theorem_C}{Theorem BI-1}
\newtheorem*{theorem_D}{Theorem BI-2}
\newtheorem*{theorem_C'}{Theorem BI-3}
\newtheorem*{theorem_D'}{Theorem BI-4}

\DeclareMathOperator{\gkdim}{GK-dim}

\DeclareMathOperator{\rk}{rk}

\DeclareMathOperator{\kk}{\mathbb{K}}

\begin{document}
\setcounter{page}{1} 

\baselineskip .65cm 
\pagenumbering{arabic}
\title[Ambiskew polynomial rings]{   Bernstein-type inequalities for quantum algebras  }
\author [Sanu Bera~Ashish Gupta~Sugata Mandal~and~Snehashis Mukherjee]{Sanu Bera$^1$, Ashish Gupta$^2$, Sugata Mandal$^3$, Snehashis Mukherjee$^4$}
\address {\newline  Sanu Bera$^1$, \newline Gandhi Institute of Technology and Management (GITAM), Hyderabad, \newline Rudraram, Patancheru Mandal. Hyderabad-502329. Telangana, India. 
\newline Ashish Gupta $^2$, \newline School of Mathematical Sciences, \newline Ramakrishna Mission Vivekananda Educational and Research Institute (rkmveri), \newline Belur Math, Howrah-711202, West Bengal, India.
\newline Sugata Mandal$^3$, \newline School of Mathematical Sciences, \newline Ramakrishna Mission Vivekananda Educational and Research Institute (rkmveri), \newline Belur Math, Howrah-711202, West Bengal, India. 
\newline Snehashis Mukherjee$^4$, 
\newline Indian Institute of Technology (IIT), Kanpur, \newline Kalyanpur. Kanpur -208 016, Uttar Pradesh, India.
 }
\email{\href{mailto:sanubera6575@gmail.com}{sanubera6575@gmail.com$^1$};\href{mailto:a0gupt@gmail.com}{a0gupt@gmail.com$^2$};\newline \href{mailto:gmandal1961@gmail.com}{gmandal1961@gmail.com$^3$};\href{mailto:tutunsnehashis@gmail.com}{tutunsnehashis@gmail.com$^4$}}
\begin{abstract}
 %The Gelfand-Kirillov dimension (GK-dimension) has emerged as a crucial tool in the study of non-commutative infinite dimensional algebras and their modules.
 We establish Bernstein-type inequalities for the quantum algebras
 $K_{n,\Gamma}^{P,Q}(\kk)$ introduced by K. L. Horton that include the graded quantum Weyl algebra, the quantum symplectic space, the quantum Euclidean space, and quantum Heisenberg algebra etc., obtaining new results and as well as simplified proofs of previously known results.
 The Krull and global dimensions of certain further localizations of $K_{n,\Gamma}^{P,Q}(\kk)$ are computed. 
\end{abstract}
\maketitle
\section{Introduction}\label{introduction}
%This is well-known from the theory of modules over the classical Weyl algebras $A_n(\mathbb{C})$ where 
%an important result (Bernstein's inequality) 
The Gelfand--Kirillov dimension of a nonzero finitely generated $A_n(\mathbb{C})$-module is bounded below by $n = \frac{1}{2}(\gkdim(A_n(\mathbb{C}))$ where $A_n(\mathbb{C})$ is the $n$-th Weyl algebra over the field $\mathbb{C}$ -  a fact commonly known as the Bernstein's inequality (B.I) \cite{bi}.   It plays a key role in the theory of $D$-modules (\cite{sc}).
\par The quantum Weyl algebras  $A_n^{\overline{q},\Lambda}$ appeared in the work of Maltsiniotis \cite{gm} on noncommutative differential calculus and are generally regarded as $q$-analogues of the classical Weyl algebras. 
%We recall the definition of $A^{\overline{q},\Lambda}_{n}$ in \cite{dj}. 
We recall the definition and some basic facts:
let $\Lambda:=\left(\lambda_{ij}\right)$ be an  $n \times n$ multiplicatively antisymmetric matrix over $\mathbb{K}$, that is, $ \lambda_{ii}=1$ and $\lambda_{ij}\lambda_{ji}=1$ for all $1 \leq i,j\leq n$ and let $\overline{q}:=(q_1,\cdots,q_n)$ be an $n$-tuple of elements of $\mathbb{K}\setminus\{0,1\}$. Given such $\Lambda$ and $\overline{q}$, the $n$-th {quantized Weyl algebra} $A^{\overline{q},\Lambda}_{n}$ is the $\mathbb{K}$-algebra generated by the variables $x_1,\cdots,x_n,y_1,\cdots,y_n$ subject to the following relations:
\begin{align*}
    x_ix_j&=q_i\lambda_{ij}x_jx_i&& i<j,\\
    y_iy_j&=\lambda_{ij}y_jy_i && i<j,\\
    x_iy_j&=\lambda_{ij}^{-1}y_jx_i &&  i < j,\\  x_iy_j&=q_j\lambda_{ij}^{-1}y_jx_i && i > j, \\
    x_{i}y_i-q_iy_ix_i&= 1 + \sum_{l=1}^{i-1}(q_l-1)y_lx_l && \forall\  i.
\end{align*}
%introd
It is well known (e.g., \cite{dj}) that the algebra $A^{\overline{q},\Lambda}_{n}$  admits a presentation as an iterated skew polynomial ring twisted by automorphisms and derivations and hence it is an affine noetherian domain. 
The family of ordered monomials $\{y_1^{a_1}x_1^{b_1}\cdots y_n^{a_n}x_n^{b_n}:a_i,b_i\geq 0\}$ forms a $\mathbb{K}$-basis.  By \cite[Section 2.8]{dj}, we know that the elements 
\begin{center}
    $z_i=1+\sum\limits_{j=1}^{i}(q_i-1)y_jx_j, \ \ 1 \leq i \leq n$
\end{center}
are normal in $A^{\overline{q},\Lambda}_{n}$ and satisfy the commutation relations
\begin{equation}\label{relation-zi}
    z_iz_j=z_jz_i \quad   \text{for all\ \ $1\leq i,j\leq n$}.
\end{equation}
 Moreover,
\begin{equation}\label{relation-yi}
    z_jy_i=\begin{dcases}y_iz_j \quad
    &\text{when $j <i$}\\
    q_iy_iz_j \quad &\text{when $j \geq i$}
\end{dcases}    
\end{equation} and 
\begin{equation}\label{relation-xi}
    z_jx_i=\begin{dcases}x_iz_j \quad &\text{when $j <i$}\\
    q_i^{-1}x_iz_j \quad  &\text{when $j \geq i$.}
\end{dcases}
\end{equation}
For the quantized weyl algebra $A_n^{\overline{q},\Lambda}$ it was shown in \cite{nf,rigal,ed} that Bernstein's inequality holds in fact for a certain simple localization $B_n^{\overline{q},\Lambda}$ of $A_n^{\overline{q},\Lambda}$ and consequently holds for a certain subclass of $A_n^{\overline{q},\Lambda}$-modules, namely the $\mathcal{Z}$-torsionfree $A_n^{\overline{q},\Lambda}$-modules (see \cite[Section 2.8]{dj}). Here $\mathcal Z$ stands for the multiplicative Ore subset generated by the $z_i$.
\begin{theorem_A} \label{oBI}\emph{\cite[Corollary 3.2]{nf}}
Suppose that no $q_i$ is a root of unity   $1\leq i \leq n$. If $M$ be a nonzero finitely generated module over $A_n^{\overline{q},\Lambda}$ that is not a $\mathcal{Z}$-torsion module, then \[ \gkdim (M) \geq n= \frac{\gkdim({A^{\overline{q},\Lambda}_{n}})}{2}.\] 
 \end{theorem_A}
%The inequality of Bernstein is also known to hold for the quantum symplectic spaces which share a common localization with the quantum Weyl algebra (\cite{bkg}). 
In a similar sense we establish Bernstein-type inequalities for a fairly general class of algebras $K_n:= K_{n,\Gamma}^{P,Q}(\kk)$ introduced by K. L. Horton in \cite{kh} which include as particular cases the graded quantum Weyl algebras, quantum symplectic spaces, quantum Euclidean spaces, quantum Heisenberg algebra and that are considered to be the multiparameter version of the aforementioned algebras.
%that include quantum graded Weyl algebra, quantum symplectic spaces, quantum Euclidean $2n$-space and some other algebras and establish a Bernstein-type inequality for these algebras. We also present a new proof of Bernstein's inequality for the quantum Weyl algebra.
%\par Recall the algebra $K_{n,\Gamma}^{P,Q}(\kk)$ in \cite{kh}.  Let $\Gamma:=\left(\gamma_{ij}\right)$ be an  $n \times n$ multiplicatively antisymmetric matrix over $\mathbb{K}$, that is, $ \gamma_{ii}=1$ and $\gamma_{ij}\gamma_{ji}=1$ for all $1 \leq i,j\leq n$ and 
\par Let $ P:= (p_1,\ldots, p_n) $ and $ Q:=(q_1,\ldots,q_n)$ be $n$-tuples of elements of $\kk^{*}$ such that $ p_i \neq q_i\ \ \forall\ \ 1\leq i\leq n$. Given such $P,Q$ and a multiplicatively antisymmetric $n \times n$ matrix $\Gamma=(\gamma_{ij})$, we define $K_{n,\Gamma}^{P,Q}(\kk)$ as the $\mathbb{K}$-algebra generated by the variables $x_1,y_1,\ldots,x_n,y_n$ subject to the following relations:
\begin{align*}
    x_ix_j&=q_ip_{j}^{-1}\gamma_{ij}x_jx_i&& i<j,\\
    y_iy_j&=\gamma_{ij}y_jy_i && i<j,\\
    x_iy_j&=p_j\gamma_{ij}^{-1}y_jx_i &&  i < j,\\  x_iy_j&=q_j\gamma_{ij}^{-1}y_jx_i && i > j, \\
    x_{i}y_i-q_iy_ix_i&=\sum_{l=1}^{i-1}(q_l-p_l)y_lx_l && \forall\  i.
\end{align*}
Henceforth when there is no chance of confusion we shall abbreviate $K_{n,\Gamma}^{P,Q}(\kk)$ as $K_n$. 
 The algebra $K_n$ can be presented as an iterated skew polynomial ring twisted by automorphisms and derivations (e.g., \cite[Proposition 3.5]{kh}) and consequently is an affine noetherian domain. The family of ordered monomials $\{y_1^{a_1}x_1^{b_1}\cdots y_n^{a_n}x_n^{b_n}:a_i,b_i\geq 0\}$ forms a $\mathbb{K}$-basis (\cite{kh}).
Moreover, the elements $${z}_{i}:=\sum\limits_{l\leq i}(q_l- p_l) y_lx_l$$ for $1\leq i \leq n$ are normal in  $ K_n$ (Section \ref{iteration}).
Let $\mathscr{Z}:=  \{z_1^{j_1},\ldots,  z_n^{j_n}: j_i \in \mathbb{N}_{0}\}$. Clearly $\mathscr{Z}$ is an Ore subset in $K_n$. In Section \ref{localization-Kn} we will show that $\gkdim(K_n) =2n$. 
As in \cite{kh}
 we will assume throughout that $p_iq_{i}^{-1}$ is not a root of unity for each $i \in \{1,\ldots,n \}$. As  mentioned above for special values of the multiparameters we obtain the following:
\begin{enumerate}
    \item[(i)] the quantum graded Weyl algebra $A_n^{Q,\Gamma}(\mathbb{K})$  ($p_i=1$; $1\leq i\leq n$),
    \item[(ii)] the quantum symplectic space $\mathcal{O}_q(\mathfrak{sp}(\mathbb{K}^{2n}))$  ($p_i=1,\ q_i=q^{-2}$ and $\gamma_{ij}=q$ for $1\leq i,j\leq n$ with $i<j$),
    \item[(iii)] the quantum Euclidean $2n$-space $\mathcal{O}_q(\mathfrak{o}\mathbb{K}^{2n})$  ($p_i=q^{-2},\ q_i=1$ and $\gamma_{ij}=q^{-1}$ for $1\leq i,j\leq n$ with $i<j$),
    \item[(iv)] the  quantum Heisenberg space-$F_q(n)$  ($p_i=q^{2},\ q_i=1$ and $\gamma_{ij}=q$ for $1\leq i,j\leq n$ with $i<j$).
\end{enumerate}
%We present a new method for calculating the Krull dimension of certain localizations of $K_n$ and, consequently, establish a Bernstein-type inequality for certain $K_n$.
%We establish a Bernstein-type inequality for these algebras. 
%We also present a new proof of Bernstein's inequality for the quantum Weyl algebra.
%We give a new method of calculating the Krull dimension of certain localizations of $K_n$ and then establish  Bernstein-type inequality for certain $K_n$.
%We present a new method for calculating the Krull dimension of certain localizations of $K_n$ and, consequently, establish a Bernstein-type inequality for certain $K_n$.
Note that all of the above examples are of the following two types: $(a)$ $p_i=1$ and no $q_i$ is a root of unity and $(b)$ $p_i=1$ and no $q_i$ is a root of unity. We establish the following Bernstein-type inequalities corresponding to each of the above cases.  
\begin{theorem_C}
 Suppose that $p_i=1$ and no $q_i$ is a root of unity   $1\leq i \leq n$. If $M$ is a nonzero finitely generated module over $K_n$ that is not a $\mathscr{Z}$-torsion module, then \[ \gkdim (M) \geq n= \frac{\gkdim(K_n)}{2}.\]
\end{theorem_C}
\begin{theorem_D}
  Suppose that $q_i=1$ and no $p_i$ is a root of unity   $1\leq i \leq n-1$. Let $M$ be a nonzero finitely generated module over $K_n$ that is not a $\mathscr{Z}$-torsion module, then \[ \gkdim (M) \geq n-1 = \frac{\gkdim(K_n)}{2} -1.\] 
\end{theorem_D}
The following corollaries are a direct consequence of the above two theorems.
\begin{coro}[B.I. for the graded quantum Weyl algebra]
     Suppose that no $q_i$ is a root of unity   $1\leq i \leq n$. If $M$ is a nonzero finitely generated module over $A_n^{Q,\Gamma}(\mathbb{K})$ that is not a $\mathscr{Z}$-torsion module, then \[ \gkdim (M) \geq n= \frac{\gkdim(A_n^{Q,\Gamma}(\mathbb{K}))}{2}.\]
\end{coro}
\begin{coro}[B.I for the quantum symplectic space]
     Suppose that $ q$ is not a root of unity. If $M$ is a nonzero finitely generated module over $\mathcal{O}_q(\mathfrak{sp}(\mathbb{K}^{2n}))$ that is not a $\mathscr{Z}$-torsion module, then \[ \gkdim (M) \geq n= \frac{\gkdim(\mathcal{O}_q(\mathfrak{sp}(\mathbb{K}^{2n})))}{2}.\]
\end{coro}
\begin{coro}[B.I for the quantum Euclidean space]
     Suppose that $ q$ is not a root of unity. If $M$ is a nonzero finitely generated module over $\mathcal{O}_q(\mathfrak{o}\mathbb{K}^{2n})$ that is not a $\mathscr{Z}$-torsion module, then \[ \gkdim (M) \geq n= \frac{\gkdim(\mathcal{O}_q(\mathfrak{o}\mathbb{K}^{2n}))}{2} - 1.\]
\end{coro}
\begin{coro}[B.I for the quantum Heisenberg space]
    Suppose that $ q$ is not a root of unity. If $M$ is a nonzero finitely generated module over $F_q(n)$ that is not a $\mathscr{Z}$-torsion module, then \[ \gkdim (M) \geq n= \frac{\gkdim(F_q(n))}{2} - 1.\]
\end{coro}
Based on our approach we sketch a quick proof Theorem \ref{oBI} in Section \ref{quickproveweyl}.

This paper is organized as follows. 
\begin{enumerate}
\item[(i)] Section 2 recalls the notion of an ambiskew polynomial ring due to D.~Jordan~\cite{dj}.  
\item[(ii)]
In Section 3 we discuss certain localizations and further localizations for the algebras $K_n$. 
\item[(iii)] Section 4 is devoted to the preliminaries on the modules over a quantum torus (that occurs as a further localization of $K_n$) and presents some key results concerning their GK dimension.
\item[(iv)] Section 5  deals with the problem of computing the (Krull/global) dimension of quantum tori arising as localizations of $K_n$ that are used in the proof of the main results.
\item[(v)]
Section 6 presents the proofs of the 
main theorems Theorem BI-1 and BI-2 (these are Corollaries \ref{coro-for-pi1} and \ref{coro-for-qi1} respectively) and a quick proof of Theorem 1 is indicated.   
\end{enumerate}
\section{Iterated Ambieskew polynomial structure of $K_n$}
The simplicity of our approach relies on the fact that our algebras fit into the framework of an iterated ambiskew polynomial ring as defined in \cite{dj}.  We recall the construction here referring to \cite{dj} for further details. 
\subsection {The construction of ambiskew polynomial rings}
\label{ambiskew-poly-ring}
Let $ A$ be an affine domain over $ \kk $. Suppose $u$ is a normal element of $ A $ inducing the automorphism $ \gamma $ of $A$ with $ua=\gamma(a)u$ for all $a\in A$. Let $\alpha $ be a $\kk$-automorphism of $A$ commuting with $ \gamma $ and let $ \beta := \gamma \alpha^{-1}$. Let $\rho\in \mathbb{K}^{*}$. From this data a $\kk$-algebra $R=R (A,u,\alpha,\rho )$ was constructed in \cite{dj}  as follows: Let $A[x;\alpha]$ denote the skew polynomial ring over $A$. Extend $\beta$ to a $\mathbb{K}$-algebra automorphism on $A[x;\alpha]$ by setting $\beta(x)=\rho x$. Then there exists a $\beta$-derivation $\delta$ on $A[x;\alpha]$ such that $\delta(A)=0$ and $\delta(x)=u - \rho\alpha(u)$. Then
$ R $ is the iterated skew-polynomial ring $A[x;\alpha][y;\beta,\delta]$. Thus, 
\begin{equation}\label{rho}
  x a= \alpha(a) x,\quad y a= \beta(a) y,\quad yx-\rho xy = u - \rho\alpha(u),\ \ \forall\ \ a\in A.  
\end{equation}
 %Here $ \delta $ is a left $ \beta$-derivation on $A[y,\alpha]$, i.e., $\delta(rs)=\beta(r)\delta(s) + \delta(r)s \quad \forall\ \  r,s \in A[y,\alpha].$
Let us define $z:=yx-u=\rho(xy-\alpha(u))$. We can easily verify that 
\begin{equation}\label{action-in-A}
 za=\gamma(a)z \quad \text{for all $a\in A$},
 \end{equation}
 \begin{equation}\label{action-in-x}
  zx=\rho xz   
 \end{equation}
 and 
 \begin{equation}\label{action-in-y}
    zy=\rho^{-1}yz. 
 \end{equation}
 Thus $z$ is a normal element in $R$ inducing an automorphism $\theta$ of $R$ such that $\theta(a)=\gamma(a), \theta(x)=\rho x$ and $\theta(y)=\rho^{-1}y$. We shall call $z$ the Casimir element of $R$. This process can be iterated by taking $z$ as the normal element for the next stage.
\subsection{$K_n$ as an iterated ambiskew polynomial ring}\label{iteration} Now we will show that the algebra $ K_{n}$ may be constructed by iterating the construction in Section \ref{ambiskew-poly-ring}. By definition $ K_{1}$ is the $\kk$-algebra generated by $ y_1 $ and $ x_1 $ subject to the relation $x_1y_1 = q_1 y_1x_1$. It is easy to check that $ z_1 := (q_1 - p_1)y_1x_1$ is a normal element of $ K_1$ such that $ z_1 x_1 = q_1^{-1} x_1 z_1$ and $ z_1 y_1 = q_1 y_1 z_1$.  First, we wish to provide details of the construction of $ K_2=K_2(K_1,u,\alpha,\rho)$ from $K_1$. We take $u := (p_2 - q_2)^{-1}z_1$. Clearly then $ u $ is normal in $K_1$ and induces the automorphism $ \gamma $ of $K_1$ such that \[ \gamma(x_1) = q_1^{-1}x_1 \quad \text{and} \quad \gamma(y_1) = q_1 y_1.\]
Also, there is a $\kk$-automorphism $\alpha$ of $K_1$ such that \[\alpha(x_1) = q_1^{-1}p_2 \gamma_{21}x_1 \quad \text{and} \quad \alpha(y_1) = q_1 \gamma_{12} y_1.\]
Observe that $ \alpha(u) = p_2u$ and $ \alpha \gamma = \gamma \alpha $. If $ \beta = \gamma \alpha^{-1}$ then \[ \beta(x_1) = p_2^{-1} \gamma_{12} x_{1} \quad \text{ and }\quad \beta (y_1) = \gamma_{21} y_1.\]
In \eqref{rho}, take $ \rho := q_2^{-1}$ and write $ x_2$ and $ y_2 $ for $ x $ and $ y
$ respectively. Then $\beta (x_2)=q_2^{-1}x_2$ and 
        \[\delta(x_2)=u - \rho \alpha(u)=(1-q_2^{-1} p_2) u= - q_2^{-1}(q_1 - p_{1}) y_1 x_1.\]
 Thus 
\begin{align*}
       x_2 x_1 &= q_{1}^{-1} p_2 \gamma_{21} x_1 x_2, & x_2 y_1 &= q_1 \gamma_{12} y_1 x_2,\\
       y_2 y_1 &= \gamma_{21} y_1 y_2, & y_2 x_1 &= p_2^{-1} \gamma_{12} x_1 y_2,\\
       x_2y_2 -q_2 y_2 x_2 &=  (q_1 - p_{1}) y_1 x_1.&&
   \end{align*}
%Also \begin{align*}
%       y_2x_2 &= q_2^{-1} x_2 y_2 -  q_2^{-1}(q_1 - p_{1}) y_1 x_1\\
%       \implies x_2y_2 &= q_2 y_2 x_2 +  (q_1 - p_{1}) y_1 x_1
%   \end{align*} 
These are precisely the extra relations required, in addition to those of $K_1$. Then the Casimir element is \[z=y_2x_2 - u=(q_2 - p_2)^{-1}\big((q_2 - p_2)y_2x_2 + (q_1 - p_1)y_1x_1\big).\]
%\begin{align*}
%         z & = y_2x_2 - u\\
%          & = y_2 x_2 - \frac{ ( q_1 - p_1)}{(p_2 - q_2)}y_1 x_1\\
%          & = \frac{(q_2 - p_2)y_2x_2 + (q_1 - p_1)y_1x_1}{(q_2 - p_2)}.    %\end{align*}
In view of \cite[Remark 2.2]{dj} we may scale $z$, viz.
$$ z_2 :=(q_2 - p_2)z= (q_2 - p_2)y_2x_2 + (q_1 - p_1)y_1x_1.$$
 Then $ z_2$ is normal in $K_2$ and  by \eqref{action-in-A}-\eqref{action-in-y}
 \[ z_2 x_1 = q_1 ^{-1} x_1 z_1, \quad z_2 y_1 = q_1 y_1 z_1 \quad \text{and} \quad z_2 z_1 = z_1 z_2.\] Also \[z_2 x_2 = q_2^{-1} x_2 z_2 \quad \text{and} \quad z_2 y_2 = q_2 y_2 z_2.\]
Since $ \alpha(z_1) = p_2 z_1$ therefore by the Lemma of   \cite[Section 2.4]{dj}, $ z_1$ is normal in $K_2$ with $$ z_1 x_2 = p_2^{-1}x_2 z_1 \quad \text{and} \quad z_1 y_2 = p_2 y_2 z_1 .$$
Suppose that $K_m$ has been constructed such that  the elements $$z_{i}:=\sum\limits_{l\leq i}(q_l- p_l) y_lx_l, \quad 1\leq i \leq m$$ are normal and satisfy the  commutation relations \cite[Lemma 3.1]{kh} 
\begin{equation}\label{eq-zi}
    z_iz_j=z_jz_i \quad   \text{for all\ \ $1\leq i,j\leq m$},
\end{equation}
\begin{equation}\label{eq-yi}
    z_jy_i=\begin{dcases}p_iy_iz_j \quad
    &\text{when $j <i$}\\
    q_iy_iz_j \quad &\text{when $j \geq i$}
\end{dcases}    
\end{equation} and 
\begin{equation}\label{eq-xi}
    z_jx_i=\begin{dcases}p_i^{-1}x_iz_j \quad &\text{when $j <i$}\\
    q_i^{-1}x_iz_j \quad  &\text{when $j \geq i$.}
\end{dcases}
\end{equation}
Then we can construct $K_{m+1}=K_{m+1}(K_{m},u,\alpha,\rho)$ from $K_{m}$, where $\rho:=q_{m+1}^{-1}$, \[ u:=(p_{m+1} - q_{m+1})^{-1}\displaystyle \sum _{l\leq m} (q_l - p_l)y_lx_l = (p_{m+1} - q_{m+1})^{-1}z_m\]  is a 
normal element in $K_{m}$ (by the induction hypothesis) and  $\alpha$ is the $\kk$-automorphism (easily verified) of $K_{m}$ such that \[\alpha(x_i) = q_i^{-1}p_{m+1} \gamma_{m+1,i}x_i \quad \text{and} \quad \alpha(y_i) = q_i \gamma_{i,m+1} y_i,\ \ 1\leq i\leq m.\]
Thus for $1 \leq i \leq m$ 
\begin{align*}
       x_{m+1} x_i &= q_{i}^{-1} p_{m+1}\gamma_{{m+1},i} x_i x_{m+1}, & x_{m+1} y_i &= q_i \gamma_{i,{m+1}} y_i x_{m+1},\\
       y_{m+1} y_i &= \gamma_{{m+1},i} y_i y_{m+1}, & y_{m+1} x_i &= p_{m+1}^{-1} \gamma_{i,m+1} x_i y_{m+1},\\
       &&x_{m+1}y_{m+1}&-q_{m+1} y_{m+1} x_{m+1}=  \displaystyle \sum_{l=1}^{m} (q_l - p_{l}) y_l x_l.
   \end{align*}
%Also \begin{align*}
%       y_2x_2 &= q_2^{-1} x_2 y_2 -  q_2^{-1}(q_1 - p_{1}) y_1 x_1\\
%       \implies x_2y_2 &= q_2 y_2 x_2 +  (q_1 - p_{1}) y_1 x_1
%   \end{align*} 
These are precisely the extra relations required, in addition to those of $K_m$. For the Casimir element we obtain  \[z=y_{m+1}x_{m+1} - u=(q_{m+1} - p_{m+1})^{-1}\big((q_{m+1} - p_{m+1})y_{m+1}x_{m+1} + \displaystyle \sum _{l\leq m} (q_l - p_l)y_lx_l\big).\]
We define
$$z_{m+1}:=(q_{m+1} -p_{m+1})z= \displaystyle \sum _{l\leq m + 1} (q_l - p_l)y_lx_l$$ of $K_{m+1}$.
 Then, $z_{m + 1}$ is normal in $K_{m + 1}$ and by \eqref{action-in-A}, \eqref{action-in-x} and \eqref{action-in-y}  we have for $1 \leq i \leq m$
 \[ z_{m+1} x_i = q_i ^{-1} x_i z_{m+1}, \quad z_{m+1} y_i = q_i y_i z_{m+1} \quad \text{and} \quad z_{m+1} z_i = z_i z_{m+1}.\] Also \[z_{m+1} x_{m+1} = q_{m+1}^{-1} x_{m+1} z_{m+1} \quad \text{and} \quad z_{m+1} y_{m+1} = q_{m+1} y_{m+1} z_{m+1}.\]
Since $ \alpha(z_i) = p_{m+1} z_i$ therefore by Lemma of   \cite[Section 2.4]{dj}, $ z_i$ is normal in $K_{m+1}$ for $1\leq i \leq m$ with $$ z_i x_{m+1} = p_{m+1}^{-1}x_{m+1} z_i \quad \text{and} \quad z_i y_{m+1} = p_{m+1} y_{m+1} z_i .$$

This completes the inductive step and our claim is established.
%The above calculations yield an inductive proof of $z_1, \ldots, z_{m}$ are normal elements in $K_{m+1}$ satisfying \eqref{eq-zi}, \eqref{eq-yi} and \eqref{eq-xi}.\par
%\begin{align}
%z_{i}x_{j} &= p_j^{-1} x_j z_{i}, \quad  &1\leq i < j \leq n\\
%z_i x_j &= q_j^{-1} x_j z_i, \quad  &1\leq j \leq i \leq n \\
%z_i y_j &= p_j y_j z_i, \quad & 1\leq i < j \leq n \label{dr1}\\
%z_i y_j &= q_j y_j z_i, \quad & 1\leq j \leq i \leq n \\
%z_i z_j &= z_j z_i \quad & 1\leq i,j \leq n\label{dr2}.
%\end{align}
\label{iteration of K_n}
\subsection{Skew commutator formulae for $K_n$}\label{skewformula} 
Set $z_0 =0$. Then the relation between $x_i$ and $y_i$ in the definition of $K_n$ can be rewritten as \[x_iy_i-q_iy_ix_i=z_{i-1},\ \ 1\leq i\leq n.\]
Adding $(q_i-p_i)y_ix_i$ to each side gives a simpler formula for $z_i$:
\[x_iy_i-p_iy_ix_i=z_i.\]
The following identities similar to those in \cite[Section 2.6]{dj} are easily checked using induction.  
\begin{itemize}
    \item [(i)] $x_1y_1^k={q_{1}}^{k}y_1^k x_{1}$~ and~ $x_1^ky_1={q_{1}}^{k}y_1x_{1}^k$,
    \item [(ii)] $x_i^ky_i=q_i^{k}y_i x_{i}^{k} + \displaystyle\frac{q_i^k - p_i^k}{q_i - p_i}z_{i-1} x_i^{k-1},\ \ \forall\ \  i \geq 2$,
    \item[(iii)] $x_iy_i^k=q_i^{k}y_i^k x_{i} + \displaystyle\frac{q_i^k - p_i^k}{q_i - p_i} y_i^{k-1} z_{i-1},\ \ \forall\ \ i \geq 2$.
\end{itemize}
  \section{ The Localizations $\mathscr{B}_n$ and $\mathscr{C}_n$ } \label{localization-Kn}
  Certain important localizations of the $n$-th quantum multiparameter Weyl algebra $A_n^{\overline{q},\Lambda}$ which is an iterated ambiskew polynomial ring were discussed in \cite{dj}. For the algebra $K_n$ analogous localizations can be defined in a similar manner.
 In $K_n$  let $\mathscr{Z}_i:=\{z_i^{j}\}_{j \geq 0}$, $\mathscr{Y}_i:=\{y_i^{j}\}_{j \geq 0}$ and $\mathscr{X}_i:=\{x_i^{j}\}_{j \geq 0}$ for $1 \leq i \leq n$.  Each $z_i$ is a normal element and thus $\mathscr{Z}_i$ is an Ore subset in $K_n$. Clearly the product $\mathscr{Z}=\mathscr{Z}_1\cdots\mathscr{Z}_n$ is also an Ore subset in  $K_n$. The localization of $K_n$ at $\mathscr{Z}$ will be denoted by $\mathscr{B}_n$. By \cite[Lemma 1.3]{krg} and \cite[Lemma 1.4]{krg} (easily adapted to the present situation), the subsets $\mathscr{Y}_i$ and $\mathscr{X}_i$  for $i = 1, \cdots, n$ are  Ore subsets in  $K_n$.
 %The given relations between $y_i,x_i$ and $z_i$,  Based on the relation $z_i=z_{i-1}+(q_i-p_i)y_ix_i$, 
 Given any sequence $v_1,\cdots,v_j$ with  $1 \leq j \leq n$ and each $v_i=y_i$ or $x_i$, there is a localization of $\mathscr{B}_n$ obtained by inverting $v_1,\cdots,v_j$. We call such a localization a $j$-fold localization of $\mathscr{B}_n$ and denote it by $\mathscr S_j$. It follows from \cite[Section 2.5]{dj} that $\mathscr{S}_n$ is generated as a $\mathbb{K}$-algebra by $\{z_1^{\pm 1},\cdots,z_n^{\pm 1}, v_1^{\pm 1},\cdots,v_n^{\pm 1}\}$ and that the monomials $z^{a_i}_iv^{b_j}_j$ (where $a_i,b_j\in \mathbb{Z}$) form a $\mathbb{K}$-basis for $\mathscr{S}_n$. In particular, the $n$-fold localization obtained by inverting $y_1,y_2,\cdots,y_n$ will be denoted as $\mathscr{C}_n$. Then $ \mathscr{C}_{n}$ is a quantum torus of rank $2n$ with generators $ z_1^{\pm 1},\cdots, z_n^{\pm 1},y_1^{\pm1},\cdots, y_n^{\pm 1}$ and relations given in  \eqref{eq-zi} and \eqref{eq-yi} as studied by J.~C.~McConnell and J.~J.~Pettit in \cite{mp}. The matrix of multiparameters for $\mathscr{C}_n$ is as follows 
 \begin{equation}\label{the matrix for K_n}
     \Lambda({\mathscr{C}_n}):=\begin{pmatrix}
         1 & 1 & \cdots & 1 & q_1 & p_2 & \cdots & p_n\\
         1 & 1 & \cdots & 1 & q_1 & q_2 & \cdots & p_n\\
         \vdots & \vdots & \cdots & \vdots & \vdots & \vdots & \cdots & \vdots\\
         1 & 1 & \cdots & 1 & q_1 & q_2 & \cdots & q_n\\
         q_1^{-1} & q_1^{-1} & \cdots & q_{1}^{-1} & 1 & \gamma_{12} & \cdots & \gamma_{1n}\\
         p_2^{-1} & q_2^{-1} & \cdots & q_{2}^{-1} & \gamma_{21} & 1 & \cdots & \gamma_{2n}\\
         \vdots & \vdots & \cdots & \vdots & \vdots & \vdots & \cdots & \vdots\\
         p_n^{-1} & p_n^{-1} & \cdots & q_{n}^{-1} & \gamma_{n1} & \gamma_{n2} & \cdots & 1
     \end{pmatrix}.
 \end{equation}
 \begin{thm} \label{isom}
Let $\mathscr
{S}_n$ be an $n$-fold localization of $\mathscr{B}_n$ as defined above. Then the map $\theta:\mathscr{C}_n\rightarrow \mathscr{S}_n$ defined by for all $1\leq i\leq n$,
\[\theta(z_i)=z_i,\ \theta(y_i)=\begin{cases} z_ix_i^{-1},& \text{when}\ v_i=x_i\\
y_i,&\text{when}\ v_i=y_i
\end{cases}\] is an isomorphism of $\mathbb{K}$-algebras.
\end{thm}
\begin{proof}
    This is a routine check noting        \eqref{eq-zi}, \eqref{eq-xi} and \eqref{eq-yi}. 
\end{proof}
 As $K_n \subseteq \mathscr{C}_{n}$ by \emph{\cite[Lemma 8.1.13]{mcr}} it follows that $ \gkdim(K_n) \leq \gkdim(\mathscr{C}_n)=2n$. On the other hand the subalgebra generated by $y_1,\ldots, y_n,z_1,\ldots, z_n$ is a quantum affine space  contained in $K_n$, whence $\gkdim(K_n) \geq 2n$. It follows that $\gkdim(K_n)=2n$.
 
%\section{Preliminaries}\label{sec2}
%\subsection{Torsion and Torsionfree Modules} 
%Let $A$ be an algebra and $M$ be a right $A$-module and $S\subset A$ be a right Ore set. The submodule
%\[\tors_{S}(M):=\{m\in M~|~ms=0\ \text{for some}\  s\in S\}\]
%is called the $S$-torsion submodule of $M$. The module $M$ is said to be $S$-torsion if $\tors_{S}(M)=M$ and $S$-torsion free if $\tors_{S}(M)=0$. If Ore set $S$ is generated by $x\in A$, we simply say that the $S$-torsion/torsionfree module $M$ is $x$-torsion/torsionfree.
%\par A nonzero element $x$ of an algebra ${A}$ is called a {\it normal} element if $x{A}={A}x$. Clearly, if $x$ is a normal element of $A$, then the set $S=\{x^i~|~i\geq 0\}$ is an Ore set generated by $x$. The next lemma is obvious.
%\begin{lem}\label{itn}
%Suppose that $A$ is an algebra, $x$ is a normal element of $A$ and $M$ is a simple $A$-module. Then either $Mx=0$ (if $M$ is $x$-torsion) or the map $x_{M}:M\rightarrow M, m\mapsto mx$ is an isomorphism (if $M$ is $x$-torsionfree).
%\end{lem}
%The above lemma says that the action of a normal element on a simple module is either trivial or invertible.
\section{Modules over over a Quantum Torus}\label{Modules over over a Quantum Torus}
%As noted in Section 2 ``a further localization'' namely, $ C_n$ of $ A_n$ is a quantum torus. 
As seen in Section \ref{localization-Kn} ``a further localization'', namely, $\mathscr{C}_n$ of $\mathscr{A}_n$, is a quantum torus. Recall that a quantum torus $\kk_{(\lambda_{ij})}[X_1^{\pm 1}, \cdots, X_n^{\pm 1}]$ of rank $n$ is defined to be the $\kk$-algebra generated by the variables $ X_1,\ldots,X_n$ together with their inverses subject to the relations:
$$ X_iX_j=\lambda_{ij}X_jX_i \qquad 1\leq i <j\leq n.$$
The quantum tori are precisely the twisted group algebras $\mathbb K \ast A$ of a finitely generated free abelian group $A$ (written multiplicatively) over the field $\mathbb{K}$ (e.g., \cite{dp}). An algebra $\mathbb{K}\ast A$ is equipped with an injective map $a \mapsto \overline{a}$,~ for all $a \in A$ such that $\{\overline{a}: a \in A\}$ is a $\mathbb{K}$-basis. Thus $\alpha\in \mathbb{K} \ast  A$ has a unique expression $\alpha= \displaystyle\sum_{a \in A} \lambda_a \overline{a}$, where almost all $\lambda_a = 0$.  If $B \leq A$ then the subset of elements of the form  $ \displaystyle \sum_{b \in B} \lambda_b \overline{b}$ is a subalgebra that is itself a twisted group algebra of $B$ and denoted as $ \mathbb{K} \ast B$.
%In particular, each subgroup $B \subseteq A \cong \mathbb{Z}$ gives rise naturally to a subalgebra $\mathbb{K} \ast  B$ of $\mathbb K \ast A$ and
It is known that for a subgroup $B \leq A$, the subset  $\mathbb{K} \ast  B \setminus \{0\}$ is an Ore subset in $\mathbb K \ast A$ (e.g., \cite{mp}).\par We recall that if $M $ is a module over a
 ring $R$ and $X$ is an Ore subset in $R$, one speaks of the $X$-torsion submodule of $M$. Here we shall be concerned with the case when $R = \mathbb{K} \ast A$ and $X = \mathbb{K} \ast B \setminus \{0\}$ for $B\leq A$. By a slight abuse of notation we will also refer to an $X$-torsion $\mathbb{K} \ast A$-module as $\mathbb{K} \ast B$-torsion and likewise speak of a $\mathbb{K} \ast B$-torsionfree module.
\par Next we briefly recall the definition of the GK dimension for $\mathbb{K}$-algebras and their modules.
Let $\mathcal A$ be an affine $\kk$-algebra  with a finite generating set $\{a_0, a_1, \cdots, a_m (=1)\}$. Set \[ \mathcal A_0 = \kk, \quad \mathcal A_1 = \sum_{i = 0}^m \kk a_i, \quad \mathcal A_i = (\mathcal A_1)^i\ (i \ge 2).\]
The sequence $\{\mathcal A_i\}_{i = 0}^\infty$ of finite-dimensional $\kk$-subspaces of $\mathcal A$ is known as the standard finite-dimensional filtration of $\mathcal A$ with respect to the given generating set.  Let $M$ be a finitely generated $\mathcal A$-module with a finite-dimensional generating subspace $M_0$.  Then $M$ has an associated finite dimensional filtration $\{ M_i \}_{i = 0}^\infty$, where $M_i = M_0\mathcal A_i$.
Then \[ \gkdim(A) := \limsup_{n \to \infty} \log_n \dim_{\kk}(\mathcal{A}_1^{n}) \]
and
\[ \gkdim(M_A) := \limsup_{n \to \infty} \log_n \dim_{\kk}(M_n). \]
 The interested reader is referred to \cite{kl} or \cite{mcr} for more details on the GK dimension.
In the articles \cite{mp} and \cite{cb1} the following facts can be immediately deduced concerning the GK dimension of $\kk_{(\lambda_{ij})}[X_1^{\pm 1}, \cdots, X_n^{\pm 1}]$-modules.
\begin{prop}\normalfont{([C. J. B. Brookes and J. R. J. Groves])}
\label{dimfgmod}
Let $M$ be a finitely generated module over $ \kk_{(\lambda_{ij})}[X_1^{\pm 1}, \cdots, X_n^{\pm 1}] $.  Then $\gkdim(M)$ equals to
\begin{itemize}
    \item[(i)] the maximal $d$ such that $M$ is not torsion as a module over the subalgebra $\mathfrak{C}_d$ generated by some subset $\{X_{i_1}^{\pm 1}, X_{i_2}^{\pm 1}, \cdots, X_{i_d}^{\pm 1}\}$ of $2d$ generators.
    \item[(ii)] $\sup \{\rk (C): C \leq A ~~ \text{and $M$ is not torsion over $\mathbb{K} \ast C$} \}$.
\end{itemize}
% together with their inverses.  
\end{prop}
\begin{proof}
This follows from \cite[Section 2]{cb1}. 
\end{proof}
\begin{prop} \cite[Lemma 8.1.13]{mcr}\label{inequality-gk}
    Let $R$ be an affine $\mathbb{K}$-algebra and $M_R$ be finitely generated $R$-module.   If $R'$ is an affine subalgebra of $R$ and  $ M'$ is a finitely generated $R'$-submodule of $M$ then $$\gkdim(M'_{R'}) \leq \gkdim(M_R).$$ 
\end{prop}

\section{The dimension of $\mathscr{C}_n$}
  It was shown in \cite{mp}  that the Krull and the global dimensions for a quantum torus $\mathcal{A}$ coincide and so we will write $\dim(\mathcal{A})$ to denote either of these. The dimension of the quantum torus $\mathscr{C}_n$ that arises as a further localization of $K_n$ will be used in our proof of the Bernstein-type inequality for $K_n$. Hence  
  we devote this section to the computation of the dimension of $\mathscr{C}_n$ corresponding to the cases $(a)$ and $(b)$ in Section \ref{introduction}. D. Jordan has already computed the dimension of the localization $A_n^{\overline{q},\Lambda}$ for the Quantum Weyl algebra $A_n^{\overline{q},\Lambda}$ (see \cite[Section 4]{dj}). In our case, when $p_i = 1$, the matrix $\Lambda(\mathscr{C}_n)$ coincides with the matrix $\Lambda(C_n^{\overline{q},\Lambda})$. Consequently, we provide a short quick proof for computing the dimension of $C_n$ (see Corollary \ref{krull-for-weyl}).
  \par As noted in  Section \ref{Modules over over a Quantum Torus} the rank $n$ quantum torus $\mathcal{A}:=\kk_{(\lambda_{ij})}[X_1^{\pm 1}, \cdots, X_n^{\pm 1}]$ has the structure of a twisted group algebra $\mathbb{K} \ast A$ of a free abelian group $A$ of rank $n$ over the field $\mathbb{K}$.  The subgroups $B \subseteq A$ so that the corresponding subalgebra $ \mathbb{K} \ast B$ is commutative plays an important role.   
    %The subgroups $B \subseteq A$ so that the corresponding subalgebra $ \mathbb{K} \ast B$ is commutative plays an important role. 
   % Before proving Lemma \ref{kd}, we present the following result from \cite{cb1}.
    \begin{thm}\cite[Theorem A]{cb1}\label{bg-kdim}
        Let $ \mathcal{A}=\mathbb{K} \ast A$  be a quantum torus then $\dim(\mathbb{K} \ast A ) $ equals to 
        \begin{itemize}
            \item[(i)] $\sup \rk(B) $,
         where the supremum is taken over the subgroups $B \subseteq A$ for which the subalgebra $\mathbb{K} \ast B$ is commutative.
         \item[(ii)]  the cardinality of a maximal system of independent commuting monomials in $\mathcal{A}$.
        \end{itemize} 
    \end{thm}
 In particular, the algebra  $\mathscr{C}_n$ can be viewed as a twisted group algebra $\mathbb{K} \ast A$ where $ A$ is the free abelian (multiplicative) group of rank $2n$ with $\mathbb{Z}$-basis $\{\widetilde{z}_1,\ldots, \widetilde{z}_n,\widetilde{y}_1,\ldots,\widetilde{y}_n\}$.  For an abelian group $A$ by the \emph{rank} of $A$ we mean its torsionfree rank, that is, \[ \rk(A): = \dim_{\mathbb Q}(A \otimes_{\mathbb  Z} \mathbb Q).\]
Note that $\mathscr{C}_n$ contains a copy %$\overline{A}$ 
of $ A$ under the correspondence \[\widetilde{z}_1^{r_1} \ldots \widetilde{z}_n^{r_n} \widetilde{y}_1^{s_1}\ldots \widetilde{y}_n^{s_n}\longmapsto \overline{\widetilde{z}_1^{r_1} \ldots \widetilde{z}_n^{r_n} \widetilde{y}_1^{s_1}\ldots \widetilde{y}_n^{s_n}}:= z_1^{r_1} \ldots z_n^{r_n} y_1^{s_1}\ldots y_n^{s_n}.\] The monomials $\gamma z_1^{r_1}\ldots z_n^{r_n} y_1^{s_1}\ldots y_n^{s_n}$ where $\gamma \in {\mathbb{K}}^{\ast}$ constitute  the group of trivial units of $ \mathbb{K} \ast  A$  and this group is nilpotent of class $2$ (\cite[Chapter 1]{dp}).  As a consequence, we have an alternating  bicharacter $\lambda: A \times A  \rightarrow \mathbb K^\ast$ defined by 
\begin{equation*}
    \lambda(a, a') = [ \overline{a}, \overline{a'}]\ \ \text{for}\ \   a, a'  \in A,
\end{equation*} satisfying
\begin{equation}\label{lmbdadefn}
     \lambda(a,a'a'')=\lambda(a,a')\lambda(a,a'') \quad \text{and} \quad
     \lambda(a',a)={\lambda(a,a')}^{-1}.
\end{equation}
%\begin{proof}
  % Since $p_i = 1$ for all $1 \leq i \leq n$, the matrix of $\mathfrak{C}_n$ in \ref{the matrix for K_n} coincides with the matrix in \ref{the matrix for C_n}. The remainder of the proof now follows directly from Theorem \ref{kd}.
%\end{proof}
\subsection{The dimension of $\mathscr{C}_n$ when $p_i=1$}
In this subsection, we will study the case when $p_i=1$.
In this case, the matrix relation for $\mathscr{C}_n$ is as follows.
\begin{equation}\label{the matrix for K_n when p_i=1}
     \Lambda({\mathscr{C}_n}):=\begin{pmatrix}
         1 & 1 & \cdots & 1 & q_1 & 1 & \cdots & 1\\
         1 & 1 & \cdots & 1 & q_1 & q_2 & \cdots & 1\\
         \vdots & \vdots & \cdots & \vdots & \vdots & \vdots & \cdots & \vdots\\
         1 & 1 & \cdots & 1 & q_1 & q_2 & \cdots & q_n\\
         q_1^{-1} & q_1^{-1} & \cdots & q_{1}^{-1} & 1 & \gamma_{12} & \cdots & \gamma_{1n}\\
        1 & q_2^{-1} & \cdots & q_{2}^{-1} & \gamma_{21} & 1 & \cdots & \gamma_{2n}\\
         \vdots & \vdots & \cdots & \vdots & \vdots & \vdots & \cdots & \vdots\\
         1 & 1 & \cdots & q_{n}^{-1} & \gamma_{n1} & \gamma_{n2} & \cdots & 1
     \end{pmatrix}.
 \end{equation}
 %This matrix $\Lambda({\mathfrak{C}_n})$ is coincide with the matrix $\Lambda({C_n})$ in \eqref{the matrix for C_n}. Thus $\mathfrak{C}_n \cong {C}_n$.
 \iffalse
 \subsection{Some important examples of $K_n$ when $p_i=1$}
\begin{example} \textbf{Graded Quantized Weyl Algebra $A_n^{Q,\Gamma}(\mathbb{K})$:} This algebra is the graded version of the algebra $A_n^{\overline{q},\Lambda}$ described earlier. In the definition of $K_n$, if we take $p_i = 1$ for all $i$, and impose no further restrictions on $\Gamma$, the relations for $K_n$ reduce to those of the algebra $A_n^{Q,\Gamma}(\mathbb{K})$. \end{example}
\begin{example} \textbf{Quantum Symplectic Space $\mathcal{O}_q(\mathfrak{sp}(\mathbb{K}^{2n}))$:} The algebra $\mathcal{O}_q(\mathfrak{sp}(\mathbb{K}^{2n}))$ was introduced by Faddeev, Reshetikhin, and Takhtadzhyan in \cite{frt}, and further relations were given by Musson in \cite{musson}. Oh explored primitive ideals in \cite{oh}, while Gomez-Torrecillas, El Kaoutit, and Benyakoub stratified the spectra of this algebra using a rank $n$ torus in \cite{gomez}. In $K_n$, for a given $q \in \mathbb{K}^*$, if we set $q_i = q^{-2}$ for all $i$ and $p_j = 1$ for all $j$, and further set $\gamma_{ij} = q$ for $i < j$, this yields the $\mathbb{K}$-algebra $\mathcal{O}_q(\mathfrak{sp}(\mathbb{K}^{2n}))$. \end{example}
\fi
 \begin{thm}\label{kd-when-pi=1}
Suppose that $p_i=1$ and $q_i$ is not a root of unity for $1\leq i \leq n$. Then $\dim(\mathscr{C}_n)=n$.    
\end{thm}
\begin{proof}
     We observe that $ z_1,\ldots,z_n$ are independent commuting monomials in $\mathscr{C}_n$. By Theorem \ref{bg-kdim} it follows that $\dim(\mathscr{C}_n) \geq n$.  Suppose  that $\dim(\mathscr{C}_n) \geq n+1$, then by another application of Theorem \ref{bg-kdim} there exists a subgroup $B \leq A$ of rank $n+1$ such that $ \mathbb{K} \ast B$ is commutative, that is, $\lambda(b_1,b_2)=1$ for all $b_1,b_2\in B$. Let $Z$ and $Y$ be the subgroups of $A$  generated by $\{\widetilde{z}_1,\ldots , \widetilde{z}_n\} $ and $\{\widetilde{y}_1,\ldots , \widetilde{y}_n\} $  respectively.  Now we define a map $\eta : Z \rightarrow \{0,1,\ldots,n\}$   by  $\eta(1)=0$ and for  $1 \neq \widetilde{z}_{1}^{k_1}\ldots \widetilde{z}_{i}^{k_i}\ldots\widetilde{z}_{n}^{k_n}$
    \[\eta(\widetilde{z}_{1}^{k_1}\ldots \widetilde{z}_{i}^{k_i}\ldots\widetilde{z}_{n}^{k_n})=\max\{i:k_i\neq 0\}.\]
    %$$\overline{z}_{j_1}^{k_1}\ldots \overline{z}_{j_m}^{k_m} \mapsto j_m,$$ where $0 \neq j_1 < j_2 < \ldots j_m$. 
\noindent \textsc{Step 1:}  Since $\rk(Z)= n$ it follows that $\rk( B \cap Z) \geq 1$. Let $$ i_1:= \max\{ \eta(\mu): \mu \in B \cap Z\}.$$ Note that $i_1 \geq 1$. Pick $\mu_1 \in B  \cap 
  Z $ such that $\eta(\mu_1)= i_1$. Set $\mu_1= \widetilde{z}_1^{r_{1,1}}\ldots \widetilde{z}_{i_1}^{r_{1,i_1}}~(r_{1,i_1}\neq 0$). Let $B_1$ be a (virtual) complement of the subgroup $\langle\mu_1\rangle$ in $B$. Then $\rk(B_1)= n$. Let $ A_{i_1}$ be the subgroup of $A$ generated by $ Z$ and $\widetilde{y}_{i_1}$. Again as  $\rk( A_{i_1})= n+1$ thus $\rk(B_1 \cap A_{i_1}) \geq 1$. Suppose $\gamma_1=\widetilde{z}_1^{p_{1,1}}\ldots \widetilde{z}_n^{p_{1,n}} \widetilde{y}_{i_1}^{s_{1,1}} \in B_1 \cap A_{i_1} $. If $s_{1,1} \neq 0$ then in view of  \eqref{lmbdadefn}, \eqref{relation-zi} and \eqref{relation-yi}  the equality $\lambda\left(\mu_1,\gamma_1\right)=1$ means
\[1=\lambda(\widetilde{z}_1^{r_{1,1}}\ldots \widetilde{z}_{i_1}^{r_{1,i_1}},\widetilde{z}_1^{p_{1,1}}\ldots \widetilde{z}_n^{p_{1,n}} \widetilde{y}_{i_1}^{s_{1,1}})=\lambda( \widetilde{z}_{i_1}^{r_{1,i_1}},\widetilde{y}_{i_1}^{s_{1,1}})=q_{i_1}^{r_{1,i_1}s_{1,1}}\]
implies that $q_{i_1}$ is a root of unity, contrary to the hypothesis. Thus we can assume that $B_1 \cap A_{i_1} \leq B\cap Z$.\\
\textsc{Step 2:} Let \[i_2:= \max\{ \eta(\mu): \mu \in B_1 \cap A_{i_1}\}.\] Pick $\mu_2\in B_1\cap A_{i_1}$ such that $\eta(\mu_2)=i_2$. Set $\mu_2= \widetilde{z}_1^{r_{2,1}}\ldots \widetilde{z}_{i_2}^{r_{2,i_2}}~(r_{2,i_2}\neq 0$). If $i_{2} < i_1$, by 
 a suitable relabeling we may assume that $i_1 < i_2$. On the other hand,  if $i_2= i_1$ then some combination $\mu_1^{j_1}\mu_2^{j_2}\in B\cap Z$ where $j_1,j_2 \in \mathbb{Z}$ satisfies $\eta(\mu_1^{j_1}\mu_2^{j_2})<i_2$  and replacing $\mu_1$ by $\mu_1^{j_1}\mu_2^{j_2}$ we may assume that $i_1 < i_2$. %We now pick a (virtual) complement $B_2$ of the subgroup  $\langle\mu_1,\mu_2\rangle$ in $B$. Thus $\rk(B_2)= n-1$. Let $A_{i_1,i_2}$ be  the subgroup of $A$ generated by $Z,\widetilde{y}_{i_1}$ and $\widetilde{y}_{i_2}$. Since $\rk(A_{i_1,i_2})=n+2$, we have $\rk(B_2\cap A_{i_1,i_2})\geq 1$. Suppose $\gamma_2=\widetilde{z}_1^{p_{2,1}}\ldots \widetilde{z}_n^{p_{2,n}} \widetilde{y}_{i_1}^{s_{2,1}}\widetilde{y}_{i_2}^{s_{2,2}} \in B_2 \cap A_{i_1,i_2}$. Then the equality  $\lambda(\mu_1,\gamma_2)=1$  means
%\[1=\lambda( \widetilde{z}_1^{r_{1,1}}\ldots \widetilde{z}_{i_1}^{r_{1,i_1}}, \widetilde{z}_1^{p_{2,1}}\ldots \widetilde{z}_n^{p_{2,n}} \widetilde{y}_{i_1}^{s_{2,1}}\widetilde{y}_{i_2}^{s_{2,2}})=\lambda(\widetilde{z}_{i_1}^{r_{1,i_1}}, \widetilde{y}_{i_1}^{s_{2,1}}\widetilde{y}_{i_2}^{s_{2,2}})= \lambda(\widetilde{z}_{i_1}^{r_{1,i_1}}, \widetilde{y}_{i_1}^{s_{2,1}})= q^{r_{1,i_1} {s_{2,1
%}}}_{i_1}, \]
  % contrary to  the fact that $q_{i_1}$  is not  a root of unity. Thus $ s_{2,1}=0$ and consequently,
  %$\gamma_2=\widetilde{z}_1^{p^2_1}\ldots \widetilde{z}_n^{p^2_n} \widetilde{y}_{i_2}^{s^2_2}$.
%Again
% the relation $\lambda(\mu_2,\gamma_2)=1$ means
%\[1 =\quad\lambda( \widetilde{z}_1^{r_{2,1}}\ldots \widetilde{z}_{i_2}^{r_{2,i_2}}, \widetilde{z}_1^{p_{2,1}}\ldots \widetilde{z}_n^{p_{2,n}} \widetilde{y}_{i_2}^{s_{2,2}}) = \lambda(\widetilde{z}_{i_2}^{r_{2,i_2}},\widetilde{y}_{i_2}^{s_{2,2}})=q_{i_2}^{{r_{2,i_2}}{s_{2,2}}}, \]
 %contrary to  the fact that  $q_{i_2}$ is not a root of unity. Thus $s_{2,2}=0$. Consequently, we obtain $B_2 \cap A_{i_1,i_2} \subset B\cap Z$. Let \[i_3:= \max\{ \eta(\mu): \mu \in B_2 \cap A_{i_1,i_2}\}.\] Pick $\mu_3\in B_2\cap A_{i_1,i_2}$ such that $\eta(\mu_3)=i_3$. 
\par Thus proceeding by induction,  there are commuting independent monomials $\mu_1,\cdots,\mu_k$ in $B\cap Z$ such that $i_1=\eta(\mu_1)<\cdots<\eta(\mu_k)=i_k$. Furthermore, let $\mu_j:= \widetilde{z}_1^{r_{j,1}}\ldots \widetilde{z}_{i_j}^{r_{j,i_j}}~(1\leq j \leq k$, $r_{j,i_j}\neq 0$). We now pick a (virtual) complement $B_k$ of the subgroup  $\langle\mu_1,\mu_2,\ldots,\mu_k\rangle$ in $B$. Thus $\rk(B_k)= n+ 1-k$. Let $A_{i_1,i_2,\ldots ,i_k}:=\langle Z,\widetilde{y}_{i_1},\ldots, \widetilde{y}_{i_k}\rangle$. Since $\rk(A_{i_1,i_2,\ldots ,i_k})=n+k$, we have $\rk(B_k\cap A_{i_1,i_2,\ldots ,i_k})\geq 1$. Suppose $ 1\neq\gamma_k=\widetilde{z}_1^{p_{k,1}}\ldots \widetilde{z}_n^{p_{k,n}} \widetilde{y}_{i_1}^{s_{k,1}}\widetilde{y}_{i_2}^{s_{k,2}}\ldots \widetilde{y}_{i_k}^{s_{k,k}} \in B_k \cap A_{i_1,i_2,\ldots,i_k}$. Let 
 $$t:=\min\{u: 1\leq u\leq k~\text{and}~ s_{k,u}\neq 0 \}.$$
 %$i_t$ denote the minimum subscript in the presentation of $\gamma_k$ such that $s_t\neq 0$. 
 Then the equality  $\lambda(\mu_t,\gamma_k)=1$ means
 \begin{align*}
   1= \lambda(\mu_t,\gamma_k)&=\lambda( \widetilde{z}_1^{r_{t,1}}\ldots \widetilde{z}_{i_t}^{r_{t,i_t}}  ,\widetilde{z}_1^{p_{k,1}}\ldots \widetilde{z}_n^{p_n} \widetilde{y}_{i_t}^{s_{k,t}}\ldots \widetilde{y}_{i_k}^{s_{k,k}})\\ &= \lambda(\widetilde{z}_{i_t}^{r_{t,i_t}},\widetilde{y}_{i_t}^{s_{k,t}}\ldots \widetilde{y}_{i_k}^{s_{k,k}})\\&=\lambda(\widetilde{z}_{i_t}^{r_{t,i_t}},\widetilde{y}_{i_t}^{s_{k,t}})\\
   &= q_{i_t}^{{r_{t,i_t}}{s_{k,t}}},
 \end{align*}
contrary to  the hypothesis that $q_{i_t}$  is  not a root of unity. It follows that  $B_k \cap A_{i_1,i_2,\ldots,i_k} \leq B\cap Z$. Let \[i_{k+1}:= \max\{ \eta(\mu): \mu \in B_k \cap A_{i_1,\ldots,i_k}\}.\] Pick $\mu_{k+1}\in B_k\cap A_{i_1,\ldots,i_k}$ such that $\eta(\mu_{k+1})=i_{k+1}$. If $i_{k+1} \notin \{i_1, \ldots, i_k\}$, by 
 a suitable relabeling we may assume that  
 \begin{equation}\label{x1}
   i_1 = \eta(\mu_1) < \ldots < i_k=\eta(\mu_k) < i_{k+1}= \eta(\mu_{k+1}).
   \end{equation}
Otherwise, if $i_{k+1} \in \{i_1, \ldots, i_k\}$, then it is clear that taking suitable combinations of the  $\mu_j$ as in Step 2 we may obtain \eqref{x1} after a relabeling of indices, if necessary.
 % by applying the procedure in Step 2, we may still assume that $i_1 < \ldots < i_k < i_{k+1}$.
    \par  Finally we obtain commuting independent monomials $\mu_1,\cdots,\mu_n$ in $B\cap Z$ such that $i_1=\eta(\mu_1)<\cdots<i_n=\eta(\mu_n)$. It follows that  the the subgroup $S:=\langle \mu_1,\cdots,\mu_n \rangle$  of $B \cap Z$ satisfies $m:=[Z:S]< \infty$. Thus $\widetilde{z}_i^m \in S$ for all $i\in\{1, \ldots , n\}$. 
    \par Now as $\rk(Y)=n$, we have $\rk(B\cap Y)\geq 1$. Suppose $1 \neq \gamma=\widetilde{y}_1^{r_1}\cdots \widetilde{y}_i^{r_i}\cdots \widetilde{y}_n^{r_n}$ in $B\cap Y$.
    Let 
 $$t:=\min\{i: 1\leq i\leq n~\text{and}~ r_{i}\neq 0 \}.$$
    %Let $t$ denote the minimum subscript $i$ in the presentation of $\gamma$ such that $r_i\neq 0$.
    Then the equality $\lambda(\widetilde{z}_t^{m},\gamma)=1$  means 
    \begin{align*}
        1=\lambda(\widetilde{z}_t^{m}, \widetilde{y}_t^{r_t}\cdots  \widetilde{y}_n^{r_n})=q_t^{mr_t}      
    \end{align*}
    implies that $q_t$ is a root of unity,  contrary to the hypothesis. This contradiction reveals that $\dim(\mathscr{C}_n)= n$.
\end{proof}
\begin{coro}
For the graded quantum Weyl algebra $A_n^{Q,\Gamma}(\mathbb{K})$ if $q_i$ is not a root of unity for $1\leq i \leq n$, then $\dim(\mathscr{C}_n)=n$. In particular, for the quantum symplectic space $\mathcal{O}_q(\mathfrak{sp}(\mathbb{K}^{2n}))$ the $\dim(\mathscr{C}_n)=n$ when $q$ is not a root of unity.
\end{coro}
\begin{coro}\label{krull-for-weyl}
 For the quantum Weyl algebra $A_n^{\overline{q},\Lambda}$ if $q_i$ is not a root of unity for $1\leq i \leq n$, then $\dim({C_n^{\overline{q},\Lambda}})=n$. 
\end{coro}
\subsection{The dimension of $\mathscr{C}_n$ when $q_i=1$}
In this subsection, we will study the case when $p_i=1$. In this case, the matrix relation for $\mathscr{C}_n$ is as follows.
\begin{equation}\label{the matrix for K_n when q_i=1}
\Lambda({\mathscr{C}_n}):=\begin{pmatrix}
         1 & 1 & \cdots & 1 & 1 & p_2 & \cdots & p_n\\
         1 & 1 & \cdots & 1 & 1 & 1 & \cdots & p_n\\
         \vdots & \vdots & \cdots & \vdots & \vdots & \vdots & \cdots & \vdots\\
         1 & 1 & \cdots & 1 & 1 & 1 & \cdots & 1\\
         1 & 1 & \cdots & 1 & 1 & \gamma_{12} & \cdots & \gamma_{1n}\\
         p_2^{-1} & 1 & \cdots & 1 & \gamma_{21} & 1 & \cdots & \gamma_{2n}\\
         \vdots & \vdots & \cdots & \vdots & \vdots & \vdots & \cdots & \vdots\\
         p_n^{-1} & p_n^{-1} & \cdots & 1 & \gamma_{n1} & \gamma_{n2} & \cdots & 1
     \end{pmatrix}.
 \end{equation}
 \iffalse
 \subsection{Some important examples of $K_n$ when $q_i=1$}
\begin{example} \textbf{Quantum Euclidean $2n$-Space $\mathcal{O}_q(\mathfrak{o}\mathbb{K}^{2n})$:} The algebra $\mathcal{O}_q(\mathfrak{o}\mathbb{K}^{2n})$ was first introduced by Faddeev in \cite{fadeev}, and a simplified set of relations was later given by Musson in \cite{musson}. In $K_n$, for $q \in \mathbb{K}^*$, if we set $q_i = 1$ and $p_j = q^{-2}$ for all $i,j$, and take $\gamma_{ij} = q^{-1}$ for $i < j$, we obtain the $\mathbb{K}$-algebra $\mathcal{O}_q(\mathfrak{o}\mathbb{K}^{2n})$. \end{example}
\begin{example} \textbf{Quantum Heisenberg Space $F_q(n)$:} By replacing $q^{-1}$ with $q$ in the previous example, we obtain the relations for the coordinate ring of the quantum Heisenberg space $F_q(n)$. This algebra was first introduced by Faddeev in \cite{fadeev} and further studied by Jakobsen and Zhang in \cite{jakobsen}, where they considered the case $\mathbb{K} = \mathbb{C}$ and $q$ a root of unity. \end{example}
\fi
\begin{thm}\label{kd-when-qi=1}
Suppose that $q_i=1$ and that $p_i$ is not a root of unity for $1\leq i \leq n$. Then $\dim(\mathscr{C}_n)=n+1$.    
\end{thm}
\begin{proof}
We observe that $ z_1,\ldots,z_n,y_1$ are independent commuting monomials in $\mathscr{C}_n$. By Theorem \ref{bg-kdim} it follows that $\dim(\mathscr{C}_n) \geq n+1$.  Suppose  that $\dim(\mathscr{C}_n) \geq n+2$. Then by another application of Theorem \ref{bg-kdim} there exists a subgroup $B \leq A$ of rank $n+2$ such that $ \mathbb{K} \ast B$ is commutative, that is, $\lambda(b_1,b_2)=1$ for all $b_1,b_2\in B$. Let $Z_{\overline{n}}$ and $Y_{\overline{1}}$ be the subgroups of $A$  generated by $\{\widetilde{z}_1,\ldots , \widetilde{z}_{n-1}\}$ and $\{\widetilde{y}_2,\ldots , \widetilde{y}_n\}$ respectively. Now we define a map $\eta : Z_{\overline{n}} \rightarrow \{0,1,\ldots,n-1\}$ by  $\eta(1)=0$ and for  $1 \neq \widetilde{z}_{1}^{k_1}\ldots \widetilde{z}_{i}^{k_i}\ldots\widetilde{z}_{n-1}^{k_{n-1}}$
    \[\eta(\widetilde{z}_{1}^{k_1}\ldots \widetilde{z}_{i}^{k_i}\ldots\widetilde{z}_{n-1}^{k_{n-1}})=\min\{i:k_i\neq 0\}.\]
    %$$\overline{z}_{j_1}^{k_1}\ldots \overline{z}_{j_m}^{k_m} \mapsto j_m,$$ where $0 \neq j_1 < j_2 < \ldots j_m$. 
\noindent \textsc{Step 1:} Since $\rk(Z_{\overline{n}})= n-1$ it follows that $\rk( B \cap Z_{\overline{n}}) \geq 1$. Let $$ i_1:= \min\{ \eta(\mu): 1 \neq \mu \in B \cap Z_{\overline{n}}\}.$$ Note that $ 1 \leq i_1 \leq n-1$.  Pick $\mu_1 \in B  \cap 
  Z_{\overline{n}} $ such that $\eta(\mu_1)= i_1$. Set $\mu_1=  \widetilde{z}_{i_1}^{r_{1,i_1}}\ldots \widetilde{z}_{n-1}^{r_{1,n-1}}$ ($r_{1,i_1}\neq 0$).  Let $B_1$ be a (virtual) complement of the subgroup $\langle\mu_1\rangle$ in $B$. Then $\rk(B_1)= n+1$. Let $ A_{i_1 +1,\overline{n}}$ be the subgroup of $A$ generated by $ Z_{\overline{n}}$ and $\widetilde{y}_{i_1+1}$. Again as $\rk( A_{i_1+1,\overline{n}})= n$ thus $\rk(B_1 \cap A_{i_1+1,\overline{n}}) \geq 1$. Suppose that $1 \neq\gamma_1=\widetilde{z}_1^{p_{1,1}}\ldots \widetilde{z}_{n-1}^{p_{1,n-1}} \widetilde{y}_{i_1+1}^{s_{1,1}} \in B_1 \cap A_{i_{1}+1,\overline{n}} $. If $s_{1,1} \neq 0$ then in view of  \eqref{lmbdadefn}, \eqref{eq-zi} and \eqref{eq-yi}  the equality $\lambda\left(\mu_1,\gamma_1\right)=1$ means
\[1=\lambda(\widetilde{z}_{i_1}^{r_{1,i_1}}\ldots \widetilde{z}_{n-1}^{r_{1,n-1}},\widetilde{z}_1^{p_{1,1}}\ldots \widetilde{z}_n^{p_{1,n}} \widetilde{y}_{i_1+1}^{s_{1,1}})=\lambda( \widetilde{z}_{i_1}^{r_{1,i_1}},\widetilde{y}_{i_1+1}^{s_{1,1}})=p_{i_1+1}^{r_{1,i_1}s_{1,1}}\]
implies that $p_{i_1+1}$ is a root of unity, contrary to the hypothesis. Thus we can assume that $B_1 \cap A_{i_1+1,\overline{n}} \leq B\cap Z_{\overline{n}}$.\\
\textsc{Step 2:} Let \[i_2:= \min\{ \eta(\mu): 1 \neq \mu \in B_1 \cap A_{i_1+1,\overline{n}}\}.\] Pick $\mu_2\in B_1\cap A_{i_1+1}$ such that $\eta(\mu_2)=i_2$. Set $\mu_2=  \widetilde{z}_{i_2}^{r_{2,i_2}} \ldots \widetilde{z}_{n-1}^{r_{2,n-1}}~(r_{2,i_2}\neq 0$). If $i_{2} < i_1$, by 
 a suitable relabeling we may assume that $i_1 < i_2$. On the other hand,  if $i_2= i_1$ then some combination $\mu_1^{j_1}\mu_2^{j_2}\in B\cap Z_{\overline{n}}$ where $j_1,j_2 \in \mathbb{Z}$ satisfies $\eta(\mu_1^{j_1}\mu_2^{j_2})>i_1$  and replacing $\mu_2$ by $\mu_1^{j_1}\mu_2^{j_2}$ we may assume that $i_1 < i_2$. %We now pick a (virtual) complement $B_2$ of the subgroup  $\langle\mu_1,\mu_2\rangle$ in $B$. Thus $\rk(B_2)= n-1$. Let $A_{i_1,i_2}$ be  the subgroup of $A$ generated by $Z,\widetilde{y}_{i_1}$ and $\widetilde{y}_{i_2}$. Since $\rk(A_{i_1,i_2})=n+2$, we have $\rk(B_2\cap A_{i_1,i_2})\geq 1$. Suppose $\gamma_2=\widetilde{z}_1^{p_{2,1}}\ldots \widetilde{z}_n^{p_{2,n}} \widetilde{y}_{i_1}^{s_{2,1}}\widetilde{y}_{i_2}^{s_{2,2}} \in B_2 \cap A_{i_1,i_2}$. Then the equality  $\lambda(\mu_1,\gamma_2)=1$  means
%\[1=\lambda( \widetilde{z}_1^{r_{1,1}}\ldots \widetilde{z}_{i_1}^{r_{1,i_1}}, \widetilde{z}_1^{p_{2,1}}\ldots \widetilde{z}_n^{p_{2,n}} \widetilde{y}_{i_1}^{s_{2,1}}\widetilde{y}_{i_2}^{s_{2,2}})=\lambda(\widetilde{z}_{i_1}^{r_{1,i_1}}, \widetilde{y}_{i_1}^{s_{2,1}}\widetilde{y}_{i_2}^{s_{2,2}})= \lambda(\widetilde{z}_{i_1}^{r_{1,i_1}}, \widetilde{y}_{i_1}^{s_{2,1}})= q^{r_{1,i_1} {s_{2,1
%}}}_{i_1}, \]
  % contrary to  the fact that $q_{i_1}$  is not  a root of unity. Thus $ s_{2,1}=0$ and consequently,
  %$\gamma_2=\widetilde{z}_1^{p^2_1}\ldots \widetilde{z}_n^{p^2_n} \widetilde{y}_{i_2}^{s^2_2}$.
%Again
% the relation $\lambda(\mu_2,\gamma_2)=1$ means
%\[1 =\quad\lambda( \widetilde{z}_1^{r_{2,1}}\ldots \widetilde{z}_{i_2}^{r_{2,i_2}}, \widetilde{z}_1^{p_{2,1}}\ldots \widetilde{z}_n^{p_{2,n}} \widetilde{y}_{i_2}^{s_{2,2}}) = \lambda(\widetilde{z}_{i_2}^{r_{2,i_2}},\widetilde{y}_{i_2}^{s_{2,2}})=q_{i_2}^{{r_{2,i_2}}{s_{2,2}}}, \]
 %contrary to  the fact that  $q_{i_2}$ is not a root of unity. Thus $s_{2,2}=0$. Consequently, we obtain $B_2 \cap A_{i_1,i_2} \subset B\cap Z$. Let \[i_3:= \max\{ \eta(\mu): \mu \in B_2 \cap A_{i_1,i_2}\}.\] Pick $\mu_3\in B_2\cap A_{i_1,i_2}$ such that $\eta(\mu_3)=i_3$. 
 \par Thus proceeding by induction,  there are commuting independent monomials $\mu_1,\cdots,\mu_k$ in $B\cap Z_{\overline{n}}$ such that $i_1=\eta(\mu_1)<\cdots<\eta(\mu_k)=i_k$. Furthermore, let $\mu_j:=  \widetilde{z}_{i_j}^{r_{j,i_j}} \ldots \widetilde{z}_{n-1}^{r_{j,n-1}}$ ($1\leq j \leq k$, $r_{j,i_j}\neq 0$). We now pick a (virtual) complement $B_k$ of the subgroup  $\langle\mu_1,\mu_2,\ldots,\mu_k\rangle$ in $B$. Thus $\rk(B_k)= n+2-k$. Let $A_{i_1+1,i_2+1,\ldots ,i_k+1,\overline{n}}:=\langle Z_{\overline{n}},\widetilde{y}_{i_1+1},\ldots, \widetilde{y}_{i_k+1}\rangle$. Since $\rk(A_{i_1+1,i_2+1,\ldots ,i_k+1,\overline{n}})=n+k-1$, we have $\rk(B_k\cap A_{i_1+1,i_2+1,\ldots ,i_k+1,\overline{n}})\geq 1$. Suppose $ 1\neq\gamma_k=\widetilde{z}_1^{p_{k,1}}\ldots \widetilde{z}_{n-1}^{p_{k,{n-1}}} \widetilde{y}_{i_1+1}^{s_{k,1}}\widetilde{y}_{i_2+1}^{s_{k,2}}\ldots \widetilde{y}_{i_k+1}^{s_{k,k}} \in B_k \cap A_{i_1+1,i_2+1,\ldots,i_k+1,\overline{n}}$. Let 
 $$t:=\max\{u: 1\leq u\leq k~\text{and}~ s_{k,u}\neq 0 \}.$$
 %$i_t$ denote the minimum subscript in the presentation of $\gamma_k$ such that $s_t\neq 0$. 
 Then the equality  $\lambda(\mu_t,\gamma_k)=1$ means
 \begin{align*}
   1= \lambda(\mu_{t},\gamma_k) &=  \lambda( \widetilde{z}_{i_t}^{r_{t,i_t}}\ldots \widetilde{z}_{n-1}^{r_{t,n-1}}  ,\widetilde{z}_1^{p_{k,1}}\ldots \widetilde{z}_{n-1}^{p_{k,{n-1}}} \widetilde{y}_{i_1+1}^{s_{k,1}}\widetilde{y}_{i_2+1}^{s_{k,2}}\ldots \widetilde{y}_{i_t+1}^{s_{k,t}})\\ &=  %\lambda(\widetilde{z}_{i_t}^{r_{t,i_t}},\widetilde{y}_{i_t}^{s_{k,t}}\ldots \widetilde{y}_{i_k}^{s_{k,k}})\\=&
\lambda(\widetilde{z}_{i_t}^{r_{t,i_t}},\widetilde{y}_{i_t+1}^{s_{k,t}})\\
   &=    p_{i_t+1}^{{r_{t,i_t}}{s_{k,t}}},
 \end{align*}
contrary to  the hypothesis that $p_{i_t+1}$  is  not a root of unity. It follows that  $B_k \cap A_{i_1+1,i_2+1,\ldots,i_k+1,\overline{n}} \leq B\cap Z_{\overline{n}}$. Let 
\[i_{k+1}:= \min\{ \eta(\mu):1\neq \mu \in B_k \cap A_{i_1+1,\ldots,i_k+1,\overline{n}}\}.\] 
Pick $\mu_{k+1}\in B_k\cap A_{i_1+1,\ldots,i_k+1,\overline{n}}$ such that $\eta(\mu_{k+1})=i_{k+1}$. If $i_{k+1} \notin \{i_1, \ldots, i_k\}$ then  by 
 a suitable relabeling we may assume that  
 \begin{equation}\label{x2}
   i_1 = \eta(\mu_1) < \ldots < i_k=\eta(\mu_k) < i_{k+1}= \eta(\mu_{k+1}).
   \end{equation}
   Otherwise, if $i_{k+1} \in \{i_1, \ldots, i_k\}$,  
 then it is clear that taking suitable combinations of the  $\mu_j$ as in Step 2 we may obtain \eqref{x2} after a relabeling of indices, if necessary.
 % by applying the procedure in Step 2, we may still assume that $i_1 < \ldots < i_k < i_{k+1}$.
    \par  Finally we obtain commuting independent monomials $\mu_1,\cdots,\mu_{n-1}$ in $B\cap Z_{\overline{n}}$ such that $i_1=\eta(\mu_1)<\cdots<i_{n-1}=\eta(\mu_{n-1})$. It follows that  the the subgroup $S:=\langle \mu_1,\cdots,\mu_{n-1} \rangle$  of $B \cap Z_{\overline{n}}$ satisfies $m:=[Z_{\overline{n}}:S]< \infty$. Thus $\widetilde{z}_i^m \in S$ for all $i\in\{1, \ldots , n-1\}$. 
    \par Now as $\rk(Y_{\overline{1}})=n-1$, we have $\rk(B\cap Y_{\overline{1}})\geq 1$. Suppose that $1 \neq \gamma=\widetilde{y}_2^{r_2}\cdots \widetilde{y}_i^{r_i}\cdots \widetilde{y}_n^{r_n}$ in $B\cap Y_{\overline{1}}$.
    Let 
 $$t:=\max\{i: 1\leq i\leq n~\text{and}~ r_{i}\neq 0 \}.$$
    %Let $t$ denote the minimum subscript $i$ in the presentation of $\gamma$ such that $r_i\neq 0$.
    Then the equality $\lambda(\widetilde{z}_{t-1}^{m},\gamma)=1$  means 
    \begin{align*}
        1=\lambda(\widetilde{z}_{t-1}^{m}, \widetilde{y}_2^{r_2}\cdots  \widetilde{y}_t^{r_t})=p_t^{mr_t}      
    \end{align*}
    implies that $p_t$ is a root of unity,  contrary to the hypothesis. This contradiction reveals that $\dim(\mathscr{C}_n)= n+1$.
\end{proof}
\begin{coro}
For the  Quantum Euclidean $2n$-Space $\mathcal{O}_q(\mathfrak{o}\mathbb{K}^{2n})$ and the Quantum Heisenberg Space $F_q(n)$, if $q$ is not a root of unity  then $\dim(\mathscr{C}_n)= n+1$.
\end{coro}
\section{Proofs of the main results } \label{alebraa-k-n}
\begin{lem}\label{kxy}
 Let $p_iq_i^{-1}$ not be a root of unity for each $1 \leq i \leq n$. Then any nonzero simple $\mathscr{B}_n$-module $M$ is either $\mathscr{X}_i$-torsion-free or $\mathscr{Y}_i$-torsion-free for all $1\leq i\leq n$.
\end{lem}
\begin{proof}
Clearly $x_1$ and $y_1$ are units in $\mathscr{B}_n$. Thus $M$ is both $\mathscr{X}_1$- and $\mathscr{Y}_1$-torsion-free. Therefore, we assume that $ i  \geq 2$. If possible, let $M$ be $\mathscr{X}_i$-torsion as well as $\mathscr{Y}_i$-torsion. Then we may pick $m(\neq 0)\in M$ such that $my_i=0$ and $mx_i^{k}=0$ but $mx_i^{k-1}\neq 0$ for some $k \ge 1$. Now using the skew-commutator formula (ii) in Section \ref{iteration of K_n}
\begin{align*}
    0=mx_i^{k}y_i&=m\big(q_i^{k}y_ix_i^{k}+\frac{q_i^{k}-p_i^{k}}{q_i-p_i}z_{i-1}x_{i}^{k-1}\big)\\
&= m x_i^{k-1}\kappa z_{i-1}.
\end{align*}
where $\kappa = \frac{q_i^k-p_i^k}{q_i- p_i} p_i^{1-k}$.
As $q_i^k \neq p_i^k$  and $z_{i-1}$ is an invertible it follows that $mx_i^{k-1}= 0$ and  we have a contradiction.
\end{proof} 
We note the following result~\cite{cb1} that is a key ingredient in our proof. 
\begin{thm}\emph{\cite[Theorem 2]{cb1}}\label{gk-krull}
    Let $M$ be a non-zero finitely generated module over a quantum torus $\mathcal{A}$ of rank $n$. Then $$ \gkdim(M) \geq n - \dim (\mathcal{A}).$$
\end{thm}
%In the case of algebras $K_n$ the Bernstein inequality holds for the localization $\mathscr{B}_n$ of $K_n$ provided that $q_i^{k}\neq p_i^{k}$ is not a root of unity.
% To illustrate our ideas we present below a shorter proof of Bernstein's inequality for the algebras $K_n$.\par
%Recall that the localization $\mathscr{S}_n = \mathscr{B}_n {\mathcal{V}}^{-1}$ (as introduced in {Section \ref{localization-Kn}}) is a quantum torus generated by $\{z_1^{\pm 1},\ldots,z_n^{\pm 1},v_1^{\pm 1},\ldots, v_n^{\pm 1}\}$ 
  %as an algebra over $\mathbb{K}$.
%The algebra  $\mathscr{S}_n$ can also be viewed as a twisted group algebra $\mathbb{K} \ast A$ where $ A$ is the free abelian (multiplicative) group of rank $2n$ with $\mathbb{Z}$-basis $\{\widetilde{z}_1,\ldots, \widetilde{z}_n,\widetilde{y}_1,\ldots,\widetilde{y}_n\}$.
\begin{thm}\label{thmgk}
    Let $p_iq_i^{-1}$ not be a root of unity for each $1 \leq i \leq n$ and  suppose that $\dim(\mathscr{C}_n)=d$. Let $M$ be a finitely generated $\mathscr{Z}$-torsionfree  $K_n$-module. Then \[ \gkdim (M)\geq 2n-d.\]
\end{thm}
\begin{proof}
%It is easily checked that each $\mathcal{V}_i$ is an Ore subset in $\mathscr{B}_n$ and thus
Since $M$ has a maximal proper submodule and the GK-dimension of a factor of $M$ bounds the GK-dimension of $M$ it suffices to assume that $M$ is simple.
Since the $z_i$ are  normal elements in $K_n$ for $ 1 \leq i \leq n$ and $M$ is a $\mathscr{Z}$-torsionfree  module thus $Mz=M$ for all $z \in \mathscr{Z}$. In view of \cite[Proposition 10.11]{gw} $M$ has a $\mathscr{B}_n$-module structure compatible with its right $K_n$-module structure. Therefore, clearly   $ M_{\mathscr{B}_n} $ is  simple.

Let $\mathcal{V}_i =\mathscr{X}_i$ if $M$ is $\mathscr{X}_i$-torsionfree and $\mathcal{V}_i= \mathscr{Y}_i$, otherwise\ ($1 \le i \le n$). It is easily checked that each $\mathcal{V}_i$ is an Ore subset in $\mathscr{B}_n$ and thus $\mathcal{V}:=\mathcal{V}_1\cdots \mathcal{V}_n$ is such (e.g \cite[Lemma 4.1]{kl}). By Lemma \ref{kxy} it is clear that $M$ is $\mathcal{V}$-torsionfree. %Indeed, let $ r$ be minimal with respect to $ m v_1^{s_1}\ldots v_r^{s_r}= 0$ for some $ 0\neq m \in M$. Then $ m':=m v_1^{s_1}\ldots v_{r-1}^{s_{r-1}} \neq 0$ but $m'v_r^{s_r}=0$, contradicting the fact that $M$ is $\mathcal{V}_{r}$-torsionfree. 
  Thus there is a localization $\widehat{M} := M\mathcal{V}^{-1}$ of $M$ at $\mathcal{V}$, that is a simple module over the localization $\mathscr{S}_n$ (as defined in {Section \ref{localization-Kn}}). 
  \par %Consequently, $\widehat{M} \cong \mathscr{S}_n/I$ for some non-zero right ideal of $I$ of $\mathscr{S}_n$. Thus in view of \cite[Proposition 3.15]{kl} it is clear that $\gkdim({\widehat{M}_{\mathscr{S}_n}}) \leq 2n-1$. 
    %It was shown in \cite[Theorem 4.9]{dj} that $ \dim(C_n)= n$ (in Theorem  \ref{kd} below, we give a shorter proof of this fact). 
    Since $\dim (\mathscr{C}_n)=d$ thus by {Theorem \ref{isom}} it follows that $\dim(\mathscr{S}_n)= d$. Therefore by Theorem \ref{gk-krull} it follows that $ \gkdim({\widehat{M}_{\mathscr{S}_n}}) \geq 2n-d $.   Thus \begin{equation}\label{inequality-c-n}
        \gkdim (\widehat{M}_{\mathscr{S}_n})\geq 2n-d.
    \end{equation} 
   % It is enough to show that $$ \gkdim(M_{B_n})= \gkdim (\widehat{M}_{C_n}).$$
    Now $K_n$  embeds in $\mathscr{S}_n$ as a subalgebra and $M$ is a $K_n$-submodule of ${\widehat{M}}_{\mathscr{S}_n}$. By Proposition \ref{dimfgmod}, $s: = \gkdim(\widehat{M}_{\mathscr{S}_n})$ is an integer maximal w.r.t. $\widehat{M}_{\mathscr{S}_n}$ not torsion (and hence torsionfree) as a module over the subalgebra $Q_s$ of $\mathscr{S}_n$ generated by some subset $\{z_{i_1}^{\pm 1}, z_{i_2}^{\pm 1}, \cdots, z_{i_t}^{\pm 1}, v_{j_1}^{\pm 1}, v_{j_2}^{\pm 1}, \cdots, v_{j_{s-t}}^{\pm 1}\}$ of cardinality $2s$ of $\{z_{1}^{\pm 1},  \cdots, z_{n}^{\pm 1}, v_{1}^{\pm 1},\cdots, v_{n}^{\pm 1}\}$. Let $ Q'_s$ be the subalgebra of $Q_s$ generated by  $z_{i_1}, z_{i_2}, \cdots, z_{i_t},v_{j_1}, v_{j_2}, \cdots, v_{j_{s-t}}$. Then $Q'_s$ is a quantum affine space of rank $s$ contained in $K_n$. Since $\widehat{M}_{\mathscr{S}_n}$ is $Q_s$-torsion-free hence $M$ is $Q'_s$-torsion-free.
    Pick $0 \ne m_0 \in M$ and let $M' : = m_0 Q'_s$. 
    Then $ M' \cong Q_{s}'$ as $ Q_{s}'$-modules and thus $\gkdim(M'{_{Q'_s}}) = \gkdim(Q'_s) = s$. 
    Then by applying Proposition \ref{inequality-gk} twice we obtain 
    $$s= \gkdim ({\widehat{M}}_{\mathscr{S}_n})\geq\gkdim(M_{\mathscr{B}_n}) \geq \gkdim(M_{K_n})\geq\gkdim(({M'})_{{Q'_s}})=  s, $$ 
    whence $\gkdim ({\widehat{M}}_{\mathscr{S}_n}) = \gkdim(M_{\mathscr{B}_n})= \gkdim(M_{K_n})$. 
    Combining this with  \eqref{inequality-c-n} we obtain the assertion in the theorem.  
   % $Q_s$ and  $M$ embeded in $\widehat{M}_{\mathscr{S}_n}$ (as $\mathscr{B}_n$-module),  $M$ is torsionfree as $Q'_s$-module.  Since $M$ is a simple $\mathscr{B}_n$-module, $M=m_0\mathscr{B}_n$ for some  nonzero $m_0  \in M$. Suppose $M'$ is a $Q_{s}'$-submodule of $M$ generated by $m_0$. Since $Q_{s}'$ is a subalgebra of $B_n$ and $ M'  \cong Q_{s}'$ as $ Q_{s}'$-module therefore   in view of  
 %Consequently, 
 % \[2n-d\leq s \leq  \gkdim (M_{\mathscr{B}_n})\leq \gkdim ({\widehat{M}}_{\mathscr{S}_n})= s,\]
  %whence  $$ 2n-1 \geq \gkdim (M_{\mathscr{B}_n}) =\gkdim ({\widehat{M}}_{\mathscr{S}_n})= s\geq 2n-d.$$
\end{proof}
\begin{rema}
    As is evident from the above proof in this situation for any finitely generated $\mathscr{B}_n$-module $M$ $$ \gkdim(M_{\mathscr{B}_n}) \geq 2n -d.$$
\end{rema}
%\begin{rema}
% The lower bound for the GK dimension of a nonzero finitely generated $K_n$-module in Corollary \ref{bi-for-k} depends on the dimension of its localization $\mathscr{C}_n$. The dimension of $\mathscr{C}_n$ with the hypothesis as in Theorems BI-1 and BI-2 is computed in Theorems \ref{kd-when-pi=1} and \ref{kd-when-qi=1} respectively.
%\end{rema}
\begin{coro}\label{coro-for-pi1}\normalfont({Theorem BI-1})
   Suppose that $p_i=1$ and no $q_i$ is a root of unity  ($1\leq i \leq n$). If $M$ is a nonzero finitely generated module over $K_n$ that is not a $\mathscr{Z}$-torsion, then \[ \gkdim (M) \geq n= \frac{\gkdim(K_n)}{2}.\]  
\end{coro}
\begin{proof}
  By Theorem \ref{kd-when-pi=1} in this case  the further localization $\mathscr{C}_n$ satisfies $\dim(\mathscr{C}_n)=n$. 
  The corollary is now immediate from Theorem \ref{thmgk}.
\end{proof}
\begin{coro} \normalfont({Theorem BI-2})\label{coro-for-qi1}
     Suppose that $q_i=1$ and no $p_i$ is a root of unity   ($1\leq i \leq n$). If $M$ is a nonzero finitely generated module over $K_n$ that is not  $\mathscr{Z}$-torsion, then \[ \gkdim (M) \geq n-1 = \frac{\gkdim(K_n)}{2} -1.\] 
\end{coro}
\begin{proof}
   By Theorem \ref{kd-when-qi=1} in this case  the further localization $\mathscr{C}_n$ satisfies $\dim(\mathscr{C}_n)=n+1$.
\end{proof}
\subsection{A simpler proof of the Bernstein inequality for the Quantum Weyl algebras $A_n^{\overline{q},\Lambda}$}\label{quickproveweyl}
We sketch a  proof of the Bernstein inequality for the quantum Weyl algebras (Theorem 1) based on our ideas.  
By \cite[Section 2.8]{dj}  the algebras $A_n^{\overline{q},\Lambda}$ are iterated ambiskew polynomial rings and a suitable analog of the skew commutator formulae  of Section \ref{skewformula} holds true (\cite[Section 2.6]{dj}). 
Moreover localizations $B_n^{\overline{q},\Lambda}$, $S_n^{\overline{q},\Lambda}$ and $C_n^{\overline{q},\Lambda}$ of $A_n^{\overline{q},\Lambda}$ corresponding to $\mathscr B_n$, $\mathscr S_n$ and $\mathscr C_n$  exist. Thus a simple $A_n^{\overline{q},\Lambda}$-module may be embedded in a simple $S_n^{\overline{q},\Lambda}$-module and the quantum torus $S_n^{\overline{q},\Lambda} \cong C_n^{\overline{q},\Lambda}$ has dimension $n$ by Corollary \ref{krull-for-weyl}. Then arguing as in the proof of  Theorem \ref{thmgk}  for a suitable quantum affine space $Q$ of rank $n$ the $A_n^{\overline{q},\Lambda}$-module $M$ embeds  $Q_{Q}$, whence $\gkdim(M) \ge \gkdim(Q) = n$. The rest is clear.
 %and the localizations $\mathscr B_n$, $
%\mathscr S_n$ and \C_n exist. \cite{dj}.  Consequently  a simple $A_n$-module may be embeded in a simple $S_n$-module and as shown in $S_n \cong C_n$ has dimension $n$ (Theorem \ref{kd-when-pi=1}). Then as 
% Although $A^{\overline{q},\Lambda}_{n}$ does not fall in the class of algebras 
% By  Corollary \ref{krull-for-weyl} we have $ \dim(C_n)=n$. Now remainder of the proof follows a similar approach to that of Lemma \ref{lemmm} and Corollary \ref{bi-for-k}  

\end{document}